\documentclass[reqno,11pt]{amsart}

\usepackage{mathrsfs}
\usepackage{amsmath,amssymb,slashed}
\usepackage{color,xcolor}
\usepackage[english]{babel}
\usepackage[utf8]{inputenc}
\usepackage{amsmath,amscd}
\usepackage{amsfonts}
\usepackage{amsthm}
\usepackage{graphicx}
\usepackage[colorinlistoftodos]{todonotes}
\usepackage{amsmath,amssymb,amsfonts}
\usepackage{underscore}
\usepackage{xy}
\usepackage[T1]{fontenc}
\usepackage[utf8]{inputenc}
\usepackage{amsthm}
\usepackage{mathrsfs}
\usepackage{amsmath,amssymb,slashed}
\usepackage{color,xcolor}
\usepackage{amsmath,amscd}

\input xy
\xyoption{all}

\theoremstyle{plain}
\newtheorem{thm}{Theorem}[section]

\newtheorem{lem}[thm]{Lemma}

\newtheorem{eg}[thm]{Example}
\newtheorem{pro}[thm]{Problem}

\theoremstyle{definition}
\newtheorem{defn}[thm]{Definition}

\theoremstyle{remark}
\newtheorem{rmk}[thm]{Remark}

\begin{document}

\title[$p$-nuclearity of reduced group $L^p$-operator algebras]{$p$-nuclearity of reduced group $L^p$-operator algebras}

\author[Z. Wang]{Zhen Wang}
\address{School of Mathematics\\HangZhou Normal University \\HangZhou 311121\\
P. R. China}
\email{wangzhen@hznu.edu.cn}



\begin{abstract}
Let $p\in (1,\infty)$ and let $G$ be a discrete group. G. An, J.-J. Lee and Z.-J. Ruan introduced $p$-nuclearity for $L^p$-operator algebras.
They proved that the reduced group $L^p$-operator algebra $F^p_\lambda(G)$ is $p$-nuclear if $G$ is amenable. In this paper, we show that the converse is true. This answers an open problem concerning the $p$-nuclearity for reduced group $L^p$-operator algebras of N. C. Phillips.
\end{abstract}

\date{\today}
\maketitle


\section{Introduction}

For $p\in[1,\infty)$, we say that a Banach algebra $A$ is an {\it $L^{p}$-operator algebra} if
it is isometrically isomorphic to a norm closed subalgebra of the algebra $\mathcal{B}(E)$ of all bounded linear operators on some $L^p$-space $E$.
Clearly, $L^p$-operator algebras {are} a natural generalization of
operator algebras on Hilbert spaces (and in particular $C^*$-algebras) by replacing Hilbert spaces with $L^p$-spaces.

The research on $L^p$-operator algebra has a long History.
It begins with the C. Herz's influential work on harmonic analysis of group algebras on $L^p$-spaces in 1970's \cite{Herz,Herz1,Herz2}. Let $G$ be a locally compact group.
C. Herz introduced the Banach algebra $PF_p(G)$ which is constructed from the
left regular representation of $G$ on $L^p(G)$.
The Banach algebra $PF_p(G)$ is called $p$-pseudofunctions of $G$ by C. Herz. This algebra is also called the reduced group $L^p$-operator algebra of $G$ and it is denoted by $F^p_\lambda(G)$ in \cite{Gardella and Thiel}. If $p=2$, then $F^2_\lambda(G)$ is the reduced group $C^*$-algebra of $G$, which is usually denoted by $C^*_\lambda(G)$.
We will use $F^p_\lambda(G)$ in the following of  this paper.

Recently, people renew interest in $L^p$-operator algebras due to the work of N. C. Phillips.
In the passing ten years, N. C. Phillips  introduced and studied $L^p$-operator algebras \cite{Phillips Lp Cuntz,Odp,Lp UHF,N. C. Phillips Lp, PlpsOpenQues,Phillips look like,Lp AF}. These studies encourage many authors to participate in the research of $L^p$-operator algebras. This includes the work on group $L^p$-operator algebras \cite{Gardella and Thiel}; groupoid $L^p$-operator algebras \cite{Gardella and Lupini groupoid}; $L^p$-operator crossed products \cite{Gardella and Thiel convolution,WZ,WZ2} and the $l^p$-Toeplitz algebra \cite{WW}. Although most previous investigations have been very largely focused on various examples, some recent works were undertaken in a more
abstract and systematic way \cite{Blecher and Phillips,Choi and,Gardella and Thiel isometry}. The reader is referred to \cite{Gardella modern} for more historical comments and recent developments in $L^p$-operator algebras.
Surprisingly, when $p\in[1,\infty)\setminus\{2\}$, the research on $L^p$-operator algebra has many wonderful results for rigidity problems \cite{Choi and,Chung,Gardella and Thiel iso}.

Nuclearity is an important property for $C^*$-algebras. This property was introduced by Takesaki \cite{Takesaki}.
A $C^*$-algebra $A$ is called {\it nuclear} if for any $C^*$-algebra $B$ there is a unique norm on the algebraic tensor product $A\odot B$. By the remarkable work of Lance \cite{Lance}, Choi-Effros \cite{Choi and Effros} and Kirchberg \cite{Kirchberg}, the nuclearity is equivalent to {\it completely positive approximation property}, that is,
there exist nets of contractive completely positive maps  $\varphi_{\alpha}:A\rightarrow M_{n(\alpha)}$ and $\psi_\alpha:M_{n(\alpha)}\rightarrow A$ such that $$\|\psi_\alpha\circ\varphi_\alpha(a)-a\|\rightarrow 0$$ for all $a\in A$.
Also it is well known that the nuclearity is equivalent to the amenability for $C^*$-algebras \cite{Connes,Haagerup}, which was originally introduced by B. E. Johnson \cite{Johnsom} for Banach algebras.

In this paper, we mainly consider $p$-nuclearity of reduced group $L^p$-operator algebras.
In \cite[Proposition 5.1(a)]{an}, G. An, J.-J. Lee and Z.-J. Ruan studied $p$-nuclearity of reduced group $L^p$-operator algebra $F^p_\lambda(G)$, and they proved this algebra $F^p_\lambda(G)$ is $p$-nuclear if $G$ is a discrete amenable group. For each positive integer $n$, we denote $M_n^p=\mathcal{B}(l^p(\{1,2,\cdots,n\},\nu))$, where $\nu$ is the counting measure on $\{1,2,\cdots,n\}$.
The following is the definition of $p$-nuclearity.

\begin{defn}\cite[Proposition 5.1(a)]{an}
Let $(X,\mathcal{B},\mu)$ be a measure space, and let $A\subset \mathcal{B}(L^p(X,\mu))$ be a norm closed subalgebra.  We say that $A$ is $p$-nuclear if there exist nets of $p$-completely contractive maps $\varphi_{\alpha}:A\rightarrow M_{n(\alpha)}^p$ and $\psi_\alpha:M_{n(\alpha)}^p\rightarrow A$ such that $$\|\psi_\alpha\circ\varphi_\alpha(a)-a\|\rightarrow 0$$ for all $a\in A$.
\end{defn}

When $p=2$ and $A$ is a $C^*$-algebra, R. R. Smith proved that $p$-nuclearity is equivalent to nuclearity (see \cite[Theorem 1.1]{Smith}).

\begin{rmk}
The $p$-nuclearity is not equivalent to the amenability of $L^p$-operator algebras. The reader is referred to \cite[Remark 1.4(iii)]{WZ2} for some examples.
\end{rmk}

\begin{eg}[{\cite{an, WZ2}}]
Let $p\in [1,\infty)$.
The following are the examples of $p$-nuclear $L^p$-operator algebras:

\begin{enumerate}
\item[(i)] $C(X)$, where $X$ is a compact metric space;
\item[(ii)]  $M_n^p$ and $\overline{\bigcup_{n=1}^\infty M_n^p}$ ;
\item[(iii)] the reduced group $L^p$-operator algebra  $F^p_\lambda(G)$, where $G$ is a discrete amenable group;
\item[(iv)]  the $L^p$-Cuntz algebra $\mathcal{O}_d^p$;
\item[(v)] the rotation $L^p$-operator algebra $F^p(\mathbb{Z},F^p(\mathbb{Z}),\beta_\theta)$ and $F^p(\mathbb{Z},S^1,\alpha_\theta)$.
\end{enumerate}

\end{eg}

Since the $F^1_\lambda(G)$ is always $1$-nuclear for discrete group $G$ (see \cite[Theorem 6.4]{an}), we only consider the following problem for $p\in (1,\infty)$.
\begin{pro}[{\cite[Problem 10.4]{PlpsOpenQues}}]\label{P:main}
Let $p\in (1,\infty)$. If $G$ is a discrete group and $F^p_\lambda(G)$ is $p$-nuclear, does it follow that $G$ is amenable?
\end{pro}

E. C. Lance proved that the reduced group $C^*$-algebra $C^*_\lambda(G)$ is nuclear if and only if $G$ is amenable for discrete group $G$ (see \cite[Theorem 4,2]{Lance} or \cite[Theorem 2.6.8]{Brown and Ozawa}).
The Arveson's Extension Theorem is very useful to prove that the nuclearity of $C^*_\lambda(G)$ implies the amenability of $G$. However, J-J. Lee gave an example that the  Arveson-Wittstock-Hahn-Banach Theorem does not hold for $p$-operator space \cite{Lee}, that is, there are $p$-operator spaces $V\subset W$, an $\mathcal{SQ}_p$ space $E$, and a $p$-completely
contractive map $\varphi: V\rightarrow\mathcal{B}(E)$ such that $\varphi$ does not extend to a $p$-completely contractive
map on $W$.
Inspired by the method of C. Anantharaman-Delaroche (see \cite[Proposition 3.5]{claire}) and G. Pisier (see \cite[Theorem 3.30]{Pisier}), we solve N. C. Phillips’ problem by proving the following theorem.

\begin{thm} \label{main thm}
Let $p\in (1,\infty)$ and let $G$ be a discrete group. The following are equivalent:
\begin{enumerate}
\item[(i)] $G$ is amenable;
\item[(ii)]  $F^p_\lambda(G)$ is $p$-nuclear;
\item[(iii)] the canonical map $F^p_\lambda(G)\stackrel{\wedge_p}{\otimes} F^p_\lambda(G)\rightarrow F^p_\lambda(G)\stackrel{\vee_p}\otimes F^p_\lambda(G)$ is an isomorphism as $p$-operator spaces;
\item[(iv)] the canonical linear map $h:F^p_\lambda(G)\odot F^p_\lambda(G)\rightarrow \mathcal{B}(l^p(G))$ given by $h\left(\lambda_p(s)\odot \lambda_p(t)\right)=\lambda_p(s)\rho_p(t)$ is continuous with respect to the $p$-operator space injective tensor norm, where $\rho_p$ is the right regular representation of $G$ and $\odot$ is the algebraic tensor product;
\item[(v)]  for any $f\in C_c(G)$, we have $\|\lambda_p(f)\|\geq |\sum_{t\in G}f(t)|$;
\item[(vi)] for any finite subset $E\subset G$, we have $|E|=\|\sum_{t\in E}\lambda_p(t)\|$.
\end{enumerate}
\end{thm}

\begin{rmk}
The condition (v) implies that the trivial representation $\mathrm{1}_G$ extends to a representation of $F^p_\lambda(G)$. This is also equivalent to the amenability of $G$ (see \cite[Theorem 5.2]{Emilie}).
\end{rmk}


The paper is organized as follows.
In Section 2, we make some preparation for the proof of Theorem \ref{main thm}.
In section 3, we give a proof of Theorem \ref{main thm}.

\section{Preliminaries}
In this section, we recall some basic notations, definitions and lemmas for the proof of Theorem \ref{main thm}.

\subsection{$p$-operator spaces}

$p$-operator spaces are closely related to $L^p$-operator algebras.
Let $p\in (1,\infty)$. A $p$-operator space is defined to be a Banach space together with a matrix norm, i.e. a norm $\|\cdot\|_n$ on each matrix space $M_n(V)$, which satisfies the following two conditions
\begin{enumerate}
\item[(i)] $\mathcal{D}_\infty: \|x\oplus y\|_{n+m}=\max \{\|x\|_n,\|y\|_n\}$ for $x\in M_n(V)$ and $y\in M_m(V)$,
\item[(ii)] $\mathcal{M}_p: \|\alpha x \beta\|_n\leq\|\alpha\|\|x\|_n\|\beta\|$ for $x\in M_n(V)$ and $\alpha,\beta\in M_n^p$.
\end{enumerate}

Let $V$ and $W$ be $p$-operator spaces. We say that a linear map $\varphi: V \rightarrow W$ is $p$-completely bounded if
$$
\|\varphi\|_{p c b}=\sup _{n\in \mathbb{Z}_{>0}}\left\{\left\|\varphi_n\right\|\right\}<\infty
$$
where $\varphi_n:\left[x_{i j}\right] \in M_n(V) \rightarrow\left[\varphi\left(x_{i j}\right)\right] \in M_n(W)$ is the induced map from $M_n(V)$ to $M_n(W)$. We say that $\varphi$ is a $p$-complete contraction (respectively, a $p$-complete isometry) if $\|\varphi\|_{p c b} \leqslant 1$ (respectively, $\varphi_n$ is an isometry for each $n \in \mathbb{Z}_{>0})$.

Let $E$ be an $L^p$-space and let $n$ be a positive integer. Then $E^n=l^p(\{1,2,\cdots n\}, E)$ with the norm $\|[x_i]\|=\left(\sum_{i=1}^n
\|x_i\|^p\right)^{\frac{1}{p}}$ is again an $L^p$-space. We can obtain a norm $\|\cdot\|_n$ on the matrix space $M_n\left(\mathcal{B}(E)\right)$ by the canonical identification $M_n\left(\mathcal{B}(E)\right)\cong \mathcal{B}(E^n)$.
Then it follows from \cite{kwa} that $\mathcal{B}(E)$ is a $p$-operator space. On the other hand, Le Merdy prove that every $p$-operator space is $p$-completely isometrically isomorphic to a norm closed subspace of $\mathcal{B}(E)$ for some $E\in \mathcal{SQ}_p$ (see \cite[Theorem 4.1]{Merdy}), where $\mathcal{SQ}_p$ is the collection of subspaces of quotients of $L^p$-spaces. The reader is referred to \cite{an,Daws1,Lee} for more research on $p$-operator spaces.

The $p$-operator space projective tensor norm $\|\cdot\|_{\wedge_p, n}$ on $M_n(V \otimes W)$ is defined by
$$
\begin{aligned}
& \|u\|_{\wedge_p, n}=\inf \{\|\alpha\|\|v\|\|w\|\|\beta\|: u=\alpha(v \otimes w) \beta \\
& \left.\quad \text { for } \alpha \in M_{n, k l}, v \in M_k(V), w \in M_l(W) \text { and } \beta \in M_{k l, n}\right\} .
\end{aligned}
$$
We let $V \stackrel{\wedge_p}\otimes W$ denote the completion of $V\otimes W$ with respect to this matrix norm, and
call $V \stackrel{\wedge_p}\otimes W$ the $p$-operator space projective tensor product of $V$ and $W$.

The following Lemma is a functorial property of the $p$-operator space projective tensor product.
\begin{lem}[{\cite[see page 938]{an}}]\label{pro func}
Let $V_1, V_2, W_1$ and $W_2$ be $p$-operator spaces. If $u_i:V_i\rightarrow W_i$ are $p$-complete contractions for
$i =1,2$, then the corresponding mapping $$u_1\otimes u_2:V_1\otimes V_2\rightarrow W_1\otimes W_2$$
extends to a $p$-complete contraction $$u_1\stackrel{\wedge_p}\otimes u_2: V_1\stackrel{\wedge_p}\otimes V_2\rightarrow W_1\stackrel{\wedge_p}\otimes W_2.$$
\end{lem}

We let $\mathcal{C B}_p(V, W)$ denote the space of $p$-completely bounded maps from $V$ to $W$. It follows from Le Merdy's characterization theorem that $\mathcal{CB}_p(V, W)$ is a $p$-operator space with the matrix norm given by $$M_n\big(\mathcal{CB}_p\left(V,W\right)\big)=\mathcal{CB}_p\big(V,M_n(W)\big).$$
In particular, the dual space $V^{\prime}=\mathcal{CB}_p(V,\mathbb{C})$ has a natural $p$-operator space structure given by $$M_n(V^{\prime})=\mathcal{CB}_p(V,M_n^p).$$

Let $V$ and $W$ be $p$-operator spaces. There exists an injective embedding
$$
\theta: x \otimes y \in V \otimes W \hookrightarrow \theta(x \otimes y) \in \mathcal{C B}_p\left(V^{\prime}, W\right)
$$
given by $\theta(x \otimes y)(f)=f(x) y$ for $f \in V^{\prime}$. The completion $V\stackrel{\vee_p} \otimes W$ of $V \otimes W$ in $\mathcal{C B}_p\left(V^{\prime}, W\right)$ is a $p$-operator subspace of $\mathcal{C B}_p\left(V^{\prime}, W\right)$. We call $V \stackrel{\vee_p}{\otimes} W$ the $p$-operator space injective tensor product of $V$ and $W$.
Let $M_m(V^{\prime})_1$ and $M_k(W^{\prime})_1$ denote the closed unit ball of $M_m(V^{\prime})$ and $M_k(W^{\prime})$, respectively.
It follows from \cite{an} that for each $u \in M_n(V \otimes W)$, the $p$-operator space injective tensor norm $\|u\|_{\vee_p, n}$ can be expressed by
$$
\|u\|_{\vee_p, n}=\sup \left\{\left\|(\varphi \otimes \psi)_n(u)\right\|: \varphi \in M_m\left(V^{\prime}\right)_1, \psi \in M_k\left(W^{\prime}\right)_1, m, k \in \mathbb{Z}_{>0}\right\}.
$$

The following Lemma is a functorial property of the $p$-operator space injective tensor product.
\begin{lem}[{\cite[see page 942]{an}}]\label{inj func}
Let $V_1, V_2, W_1$ and $W_2$ be $p$-operator spaces. If $u_i:V_i\rightarrow W_i$ are $p$-complete contractions for
$i =1,2$, then the corresponding mapping $$u_1\otimes u_2:V_1\otimes V_2\rightarrow W_1\otimes W_2$$
extends to a $p$-complete contraction
$$u_1\stackrel{\vee_p}\otimes u_2: V_1\stackrel{\vee_p}\otimes V_2\rightarrow W_1\stackrel{\vee_p}\otimes W_2.$$.
\end{lem}

\subsection{Spatial $L^p$-operator tensor products and $p$-completely bounded maps of $L^p$-operator algebras}
Let $p\geq 1$.
Let $(X,\mu)$ and $(Y,\nu)$ be two measure space, there is an $L^p$-tensor product $L^p(X,\mu)\otimes_p L^p(Y,\nu)$ which can be canonical identified with $L^p(X\times Y,\mu\times\nu)$ via $\xi\otimes\eta(x,y)=\xi(x)\eta(y)$ for all $\xi\in L^p(X,\mu)$ and $\eta\in L^p(Y,\nu)$. If $a\in \mathcal{B}(L^p(X,\mu))$ and $b\in \mathcal{B}(L^p(Y,\nu))$, then there is a corresponding tensor product operator $a\otimes b\in \mathcal{B}(L^p(X\times Y,\mu\times \nu))$. Let $A\subset \mathcal{B}(L^p(X,\mu))$ and $B\subset \mathcal{B}(L^p(Y,\nu))$ be two norm closed subalgebra. Define the algebra $A\otimes_p B\subset \mathcal{B}(L^p(X\times Y, \mu\times \nu))$ to be the closed linear span of all $a\in A$ and $b\in B$. Then $A\otimes_p B$ is an $L^p$-operator algebra, and it is called the spatial $L^p$-operator tensor product of $A$ and $B$.

\begin{rmk}
 Let $A\subset \mathcal{B}\left(L^p(X,\mu)\right)$ and $B\subset \mathcal{B}\left(L^p(Y,\nu)\right)$ be $L^p$-operator algebras. Then it follows from \cite[Theorem 3.3]{an} that $A\stackrel{\vee_p}\otimes B$ is $p$-completely isometric to $A\otimes_p B$.
\end{rmk}

Given a norm closed subalgebra $A$ of $\mathcal{B}\left(L^p(X,\mu)\right)$, the spatial tensor product $M_n^p\otimes_p A$ is the $L^p$-matrix algebra.
Clearly, each element of $M_n^p\otimes_p A$ is of form $[a_{i,j}]_{1\leq i,j\leq n}$ with $a_{i,j}\in A$,
which is also written as $\sum_{i,j=1}^n e_{i,j}\otimes a_{i,j}$, where $\{e_{i,j}\}_{1\leq i,j\leq n}$ are the canonical matrix units of $M_n^p$.

\begin{defn} 
Let $A$ be a closed subalgebra of $\mathcal{B}(L^p(X,\mu))$, $B$ be a closed subalgebra of $\mathcal{B}(L^p(Y,\nu))$
and $\varphi$ be a linear map $\varphi: A\rightarrow B$.  We denote by $\varphi_n$ the map from $M_n^p\otimes_p A$ to $M_n^p\otimes_p B$ defined by $$\varphi_n\left(\sum_{i,j=1}^n e_{i,j}\otimes a_{i,j}\right)=\sum_{i,j=1}^n e_{i,j}\otimes \varphi(a_{i,j}) $$ for
$\sum_{i,j=1}^n e_{i,j}\otimes a_{i,j}\in M_n^p\otimes_p A$.
We denote $$\|\varphi\|_{pcb}=\sup_{n\in \mathbb{Z}_{>0}}\| \varphi_n\|.$$
We say that $\varphi$ is {\it $p$-completely bounded} if $\|\varphi\|_{pcb}\leq C$ for some positive constant $C$, say that $\varphi$ is {\it $p$-completely contractive} if $\|\varphi\|_{pcb}\leq 1$, and say that $\varphi$ is {\it $p$-completely isometric} if $\varphi_n$ is isometric for all positive integer  $n$.
\end{defn}

\subsection{The regular representation and Amenability}

For a discrete group $G$, we let $\lambda_p:G\rightarrow \mathcal{B}(l^p(G))$ denote the {\it left regular representation}, that is $\lambda_p(s)(\delta_t)=\delta_{st}$ for all $s,t\in G$, where $\{\delta_t:t\in G\}\subset l^p(G)$ is the canonical basis. The {\it reduced group $L^p$-operator algebra} of $G$, denoted $F^p_\lambda(G)$, is the completion of $C_c(G)$ with respect to the norm $\|\lambda_p(f)\|.$

 There is also a {\it right regular representation} $\rho_p:G\rightarrow \mathcal{B}(l^p(G))$, defined by $\rho_p(s)\delta_t=\delta_{ts^{-1}}$.
We claim that
\begin{align}\label{lr}
\|\lambda_p(f)\|=\|\rho_p(f)\|
\end{align}
for all $f\in C_c(G)$. Let $V:l^p(G)\rightarrow l^p(G)$ be the invertible isometry given by $V\delta_t=\delta_{t^{-1}}$. Then one can check that $$V\lambda_p(f)V^{-1}\delta_t=\rho_p(f)\delta_t.$$ This will prove the claim.

Let $\sigma_p$ be the {\it conjugacy representation} of $G$ on $l^p(G)$ which is defined by $\sigma_p(s)=\lambda_p(s)\rho_p(s)$. It is easy check that $\sigma_p(s)\delta_e=\delta_e$, where $e$ is the unit of group $G$.

There are many equivalent definitions for amenable groups. We will use the following definition in the proof of the Theorem \ref{main thm}.
\begin{defn}[{\cite[Definition 11.2.3]{Dales}}]\label{amenable}
Let $p\in[1,\infty)$ and let $G$ be a discrete group. The group $G$ is amenable if there exists a net $f_\alpha\in l^p(G)$ such that $f_\alpha \geq 0$, $\|f_\alpha\|_p=1$ and $\|\lambda_p(s)f_\alpha-f_\alpha\|_p\rightarrow 0$ for all $s\in G$.
\end{defn}

\section{Proof of Theorem \ref{main thm}}

In this section, we prove the Theorem \ref{main thm}.

\begin{proof}[Proof of Theorem \ref{main thm}]
(i)$\Rightarrow$ (ii): It follows from \cite[Proposition 5.1]{an}

(ii)$\Rightarrow$ (iii): Let $\varphi_\alpha:F^p_\lambda(G)\rightarrow M_{n(\alpha)}^p$ and $\psi_\alpha:M_{n(\alpha)}^p\rightarrow F^p_\lambda(G)$ be the nets of $p$-completely contractive maps such that $$\|\psi_\alpha\circ\varphi_\alpha(a)-a\|\rightarrow 0$$ for all $a\in F^p_\lambda(G).$ Since $M_{n(\alpha)}^p$ has $p$-OAP, it follows from \cite[Theorem 3.12]{an} that $M_{n(\alpha)}^p\stackrel{\vee_p}{\otimes} F^p_\lambda(G)$ is isomorphic to $M_{n(\alpha)}^p\stackrel{\wedge_p}\otimes F^p_\lambda(G)$ and this isomorphism is denoted by $\Psi$. Let $\pi$ be the canonical map from $F^p_\lambda(G)\stackrel{\wedge_p}\otimes F^p_\lambda(G)$ to $F^p_\lambda(G)\stackrel{\vee_p}{\otimes} F^p_\lambda(G).$ By Lemma \ref{pro func} and Lemma \ref{inj func}, we have the following
$$\begin{CD}
 F^p_\lambda(G)\stackrel{\wedge_p}\otimes F^p_\lambda(G) @>\pi>>  F^p_\lambda(G)\stackrel{\vee_p}{\otimes} F^p_\lambda(G) @ >\varphi_{n(\alpha)}\stackrel{\vee_p}{\otimes} \mathrm{Id}_{F^p_\lambda(G)}>>M_{n(\alpha)}^p \stackrel{\vee_p}\otimes F^p_\lambda(G) \\
@.     @. @VV\Psi V\\
@.      F^p_\lambda(G)\stackrel{\wedge_p}\otimes F^p_\lambda(G) @<\psi_{n(\alpha)}\stackrel{\wedge_p}{\otimes} \mathrm{Id}_{F^p_\lambda(G)}<< M_{n(\alpha)}^p\stackrel{\wedge_p}\otimes F^p_\lambda(G)
\end{CD} .$$
Let $\Phi_\alpha= \left(\psi_{n(\alpha)}\stackrel{\wedge_p}{\otimes} \mathrm{Id}_{F^p_\lambda(G)}\right)\circ \Psi\circ \left( \varphi_{n(\alpha)}\stackrel{\vee_p}{\otimes} \mathrm{Id}_{F^p_\lambda(G)}\right)$. Then $\Phi_\alpha$ is a net of bounded linear maps from $F^p_\lambda(G)\stackrel{\vee_p}{\otimes} F^p_\lambda(G)$ to $F^p_\lambda(G)\stackrel{\wedge_p}\otimes F^p_\lambda(G)$.
We denote by $\mathrm{Id}_{F^p_\lambda(G)}\odot \mathrm{Id}_{F^p_\lambda(G)}:
F^p_{\lambda}(G)\odot F^p_{\lambda}(G)\rightarrow F^p_{\lambda}(G)\odot F^p_{\lambda}(G)$ the algebraic tensor product map. Since $\|\Phi_\alpha(x)-\mathrm{Id}_{F^p_\lambda(G)}\odot \mathrm{Id}_{F^p_\lambda(G)}(x)\|\rightarrow 0$ for all $x\in F^p_{\lambda}(G)\odot F^p_{\lambda}(G)$,
it follows that $\mathrm{Id}_{F^p_\lambda(G)}\odot \mathrm{Id}_{F^p_\lambda(G)}$ is a bounded linear map from $F^p_\lambda(G)\stackrel{\vee_p}{\otimes} F^p_\lambda(G)$ to $F^p_\lambda(G)\stackrel{\wedge_p}\otimes F^p_\lambda(G)$.
Hence it extends to a bounded linear map $\Phi:F^p_\lambda(G)\stackrel{\vee_p}{\otimes} F^p_\lambda(G)\rightarrow F^p_\lambda(G)\stackrel{\wedge_p}\otimes F^p_\lambda(G)$. This proves (iii).

(iii)$\Rightarrow$ (iv): We denote by $\lambda_p \cdot \rho_p$ the biregular representation $(s,t)\rightarrow \lambda_p(s)\rho_p(t)$ of $G\times G$ on $l^p(G)$. Since $\|\cdot\|_{\wedge_p}$ is the largest $p$-operator space norm \cite[Proposition 4.8]{Daws1}, it follows that the canonical linear map $$h:F^p_\lambda(G)\odot F^p_\lambda(G)\rightarrow B(l^p(G))$$ defined by $$h(\lambda_p(s)\otimes\lambda_p(t))=\lambda_p(s)\rho_p(t)$$ has a continuous extension on $F^p_\lambda(G)\stackrel{\wedge_p}\otimes F^p_\lambda(G)$ and it is denoted by $(\lambda_p\cdot\rho_p)_\lambda$. By (iii), there exists a bounded linear map from $F^p_\lambda(G)\stackrel{\vee_p}{\otimes} F^p_\lambda(G)$ to $B(l^p(G))$ which is denoted by $\widetilde{(\lambda_p\cdot\rho_p)_\lambda}$. This proves (iv).

(iv)$\Rightarrow$ (v):
We recall that $\sigma_p$ is the conjugacy representation $s\mapsto \lambda_p(s)\rho_p(s)$ of $G$ on $l^p(G)$.
By (iv), the following diagram is commutative
$$\begin{CD}
  F^p(G) @>\sigma_p>>  B(l^p(G)) \\
  @VV \lambda_p V @AA \widetilde{(\lambda_p\cdot\rho_p)_\lambda} A  \\
  F^p_\lambda(G) @>\iota>> F^p_\lambda(G))\stackrel{\vee_p}{\otimes} F^p_\lambda(G),
\end{CD}$$
 where $\iota(s)=\lambda_p(s)\otimes\lambda_p(s)$ for each $s\in G$.
Let $\theta=\widetilde{(\lambda_p\cdot\rho_p)_\lambda}\circ\iota$. Then $\sigma_p=\theta\circ\lambda_p$.

{\it Claim 1.} $\|\theta\|\leq 1$.

For any $f\in C_c(G)$ with $\|\lambda_p(f)\|\leq 1$ and $\xi\in l^p(G)$, by (\ref{lr}), we have
$$\|\theta(f)\xi\|=\|\lambda_p(f)\rho_p(f)\xi\|\leq \|\lambda_p(f)\|\cdot\|\rho_p(f)\|\cdot\|\xi\|=\|\lambda_p(f)\|^2\cdot\|\xi\|.$$
Hence $\|\theta(f)\|\leq \|\lambda_p(f)\|^2\leq 1$. It follows that $\|\theta\|\leq 1$. This proves the Claim 1.

Now we will prove (iv). Since $\sigma_p=\theta\circ\lambda_p$ and $\|\theta\|\leq 1$, it follows that $$\|\sigma_p(f)\|=\|\theta\circ\lambda_p(f)\|\leq \|\lambda_p(f)\|.$$ It is easy to check that $\sigma_p(s)\delta_e=\delta_e$. Hence $$\|\lambda_p(f)\|\geq\|\sigma_p(f)\|\geq\|\sigma_p(f)\delta_e\|=|\sum_{t\in G} f(t)|.$$ This proves (iv).

(v)$\Rightarrow$ (vi): For any finite subset $E\subset G$, by (iv), we have $\|\sum_{t\in E}\lambda_p(t)\|\geq |E|$. Obviously, $\|\sum_{t\in E}\lambda_p(t)\|\leq |E|$. This proves (vi).

(vi)$\Rightarrow$ (i): For any finite subset $E\subset G$, we can assume that $e\in E$, where $e$ is the unit of $G$. By (v),
we have $\|\frac{\sum_{t\in E}\lambda_p(t)}{|E|}\|=1$. Then there exists a sequence $(\xi_i)$ in $l^p(G)$ such that $\xi_i\geq 0$, $\|\xi_i\|_p=1$ and $\|\frac{\sum_{t\in E}\lambda_p(t)}{|E|}\xi_i\|\rightarrow 1$. Since $l^p(G)$ is a uniformly convex Banach space for $p\in (1,\infty)$, it follows from  \cite{Fan} that $l^p(G)$ is a full $k$-convex Banach space for all positive integer $k\geq 2$.
Then $$\|\lambda_p(s)\xi_i-\lambda_p(t)\xi_i\|_p\rightarrow 0$$ for all $s,t\in E$.
Since $e\in E$, it follows that $$\|\lambda_p(s)\xi_i-\xi_i\|_p\rightarrow 0$$ for all $s\in E$. Then there exists a net $(\eta_\alpha)$ in $l^p(G)$ such that $\eta_\alpha\geq 0$, $\|\eta_\alpha\|_p=1$ and $$\|\lambda_p(s)\eta_\alpha-\eta_\alpha\|_p\rightarrow 0$$ for all $s\in G$. By Definition \ref{amenable},
we have that $G$ is amenable.
\end{proof}


\section*{Acknowledgements}

The author is supported by National Natural Science Foundation of China (grant number 12201240) and Research Start-up Funding Program of Hangzhou Normal University (grant number 4235C50224204070).



\end{document}